\documentclass[reqno, 12pt]{amsart}

\usepackage[utf8]{inputenc}
\usepackage[T1]{fontenc}
\usepackage[english]{babel}
\usepackage{libertine} 
\usepackage[libertine]{newtxmath}
\usepackage{fullpage}
\usepackage{xcolor}
\usepackage[linkcolor=blue, citecolor=green, colorlinks=true]{hyperref}

\title{Murmurations using Petersson trace formula}
\author[C. I. Kuan]{Chan Ieong Kuan}
\address{School of mathematics (Zhuhai) \newline
Zhuhai Campus, Sun Yat-Sen University \newline
Tangjiawan, Zhuhai, Guangdong, 519082, China (PRC)}
\email{kuanchi3@mail.sysu.edu.cn}

\author[D. Lesesvre]{Didier Lesesvre}
\address{University of Lille \newline CNRS, UMR 8524 --- Laboratoire Paul Painlevé \newline
59000 Lille, France}
\email{didier.lesesvre@univ-lille.fr}

\date{\today}

\newtheorem{prop}{Proposition}
\newtheorem{thm}{Theorem}
\newtheorem{lem}{Lemma}

\newcommand{\GL}{\mathrm{GL}}

\newcommand{\diag}{\mathrm{diag}}
\newcommand{\main}{\mathrm{main}}
\newcommand{\err}{\mathrm{err}}
\renewcommand{\geq}{\geqslant}
\renewcommand{\leq}{\leqslant}

\begin{document}

\maketitle

\begin{abstract}
We prove the murmuration phenomenon, which is a correlation between signs of functional equations and Fourier coefficients, in the case of modular forms in the weight aspect. We in particular improve the range of visibility of murmurations compared to previous results. This is the first approach to the murmuration phenomenon using a relative trace formula, showing its robustness.
\end{abstract}

\section{Introduction}

\subsection{Murmurations}

The distributional behavior of the low-lying zeros of families of L-functions has been actively studied in the recent decades \cite{sarnak_families_2016}; it displays striking similarities with the distributional behavior of the eigenangles of classical groups of random matrices in a certain range. 
Limiting statistics for densities of low-lying zeros around the critical point (beyond which symmetry is broken between different families), studied both numerically using elliptic curve databases and theoretically using explicit formulas, features specific oscillations uncovered in recent works \cite{bober_murmurations_2023, he_murmurations_2023, zubrilina_murmurations_2023}. These are now called \textit{murmurations}.

Sarnak \cite{sarnak_letter} defined murmurations for a family $\mathcal{F}$ as follows. Denote $c(f)$ the analytic conductor of $f \in \mathcal{F}$. Let $K \geqslant 1$ be a growing parameter and $h$ a smooth compactly supported test-function. The murmuration associated with $\mathcal{F}$ with suitable weights $\omega_f$ is a function $M_h$ such that
\begin{equation}
\label{murmuration-sarnak}
\frac{\displaystyle \sum_{p \in [K-H,K+H]} \sum_{f \in \mathcal{F}} h\left( \frac{c(f)}{K} \right) \omega_f a_f(p)}{\displaystyle \sum_{p \in [K-H,K+H]} \sum_{f \in \mathcal{F}} h\left( \frac{c(f)}{K} \right) \omega_f} = M_h\left( K \right) + T(K), 
\end{equation}
where $T(K) =o(M_h(K))$ when $K$ grows to infinity. The smaller $H$ is allowed to be, the more visible the murmuration phenomenon.

This phenomenon was recently studied for elliptic curves by He et al. \cite{he_murmurations_2023}, and for modular forms by Zubrilina \cite{zubrilina_murmurations_2023} (in the level aspect) and Bober et al. \cite{bober_murmurations_2023} (in the weight aspect). All these works rely on the Selberg trace formula, and are henceforth strongly tied to the $\GL(2)$ setting. We explore in this paper an alternative proof using the relative trace formula, as already suggested in Sarnak's letter, which is now developed in explicit and quantitative form in various settings, therefore suggesting that such phenomenon could be studied for more general groups.

\subsection{Setting}

Let $H_k$ be an orthogonal basis of the family of modular forms of level $1$ and weight $k \geqslant 2$, which are Hecke eigenforms with arithmetically normalized coefficients. As suggested by the explicit formula and the study of low-lying zeros of the associated L-functions \cite{iwaniec_low_2000}, a bifurcation phenomenon does occur when averaging Fourier coefficients of modular forms over primes which are of size about the analytic conductor $k^2$. We aim at studying this phenomenon and therefore to understand the asymptotic behavior of 
\begin{equation}
\label{sigma}
\Sigma_1 := \sum_{p/K^2 \in E} \sum_{k \geqslant 2} h\left( \frac{k}{K}\right) \sum_{\substack{f \in H_k}} \omega_f \varepsilon_f \lambda_f(p) \log p,
\end{equation}
where $K \geqslant 1$ grows to infinity. Here, $E$ is a constraint set, say $E=[A,B]$ for fixed parameters $B>A>0$,  $h$ is a smooth compactly supported function (essentially truncating the weight sum to $k \leqslant K$ or $k \asymp K$), $\omega_f$ are the Petersson weights (see \cite[(2.3)]{iwaniec_low_2000}), viz.
\begin{equation}
\omega_f  := \left(\frac{\Gamma(k-1)}{(4\pi)^{k-1}}\right)^{1/2}\|f\|^{-1},
\end{equation}
 $\varepsilon_f$ is the sign of the functional equation $\varepsilon_f = i^k$ (see \cite[(3.5)]{iwaniec_low_2000}), and $\lambda_f(p)$ is the $p$-th Fourier coefficient i.e. the $p$-th Hecke eigenvalue.
 
We prove the following murmuration phenomenon.
 
 \begin{thm}
 \label{thm:thm}
Assuming Generalized Riemann Hypothesis for Dirichlet $L$-functions. Let $K$ and $M$ be paremeters satisfying $K^{1/3+\varepsilon} \ll M \ll K^{1-\varepsilon}$ and define function $h(x) := \exp(-\tfrac{(k-K)^2}{M^2})$. Let $E = [A,B]$ for $A < B$. We have, as $K$ grows to infinity, 
\begin{equation}
\label{eq:thm}
\frac{\displaystyle\sum_{p/K^2\in E} \log p \sum_{k \geqslant 2} h(k) \sum_{\substack{f \in H_k}} \omega_f \varepsilon_f \lambda_f(p) }{\displaystyle\sum_{p/K^2\in E}  \log p \sum_{k \geqslant 2} h(k) \sum_{\substack{f \in H_k}} \omega_f } \underset{
{K \to \infty}}{\sim} \frac{K^{-1}}{\sqrt{B} + \sqrt{A}}.
\end{equation}
 \end{thm}
 Note that, compared to \eqref{murmuration-sarnak}, we added the weights $\log p$ which come handily when summing over primes and applying the prime number theorem. Analogous results without these weights can be derived analogously by using partial summation.
 
Bober et al. \cite{bober_murmurations_2023} proved that, assuming the Generalized Riemann Hypothesis for both Dirichlet $L$-functions and modular forms, and letting $K^{5/6+\varepsilon} < M < K^{1-\varepsilon}$, they prove
\begin{equation}
\label{eq:thm-bober}
\frac{\displaystyle\sum_{p/K^2\in E}  \log p \sum_{{|k-K| \leq M}} \sum_{\substack{f \in H_k}}  \varepsilon_f \lambda_f(p)}{\displaystyle\sum_{p/K^2\in E}  \log p  \sum_{{|k-K| \leq M}} \sum_{\substack{f \in H_k}} 1 } \underset{{K \to \infty}}{\sim} K^{-1} \frac{\nu(E)}{|E|}, 
\end{equation}
where 
\begin{equation}
\nu(E) = \frac{1}{\zeta(2)} {\sideset{}{^*}\sum_{\substack{a, q \in \mathbb{N}^\star \\ (a,q) = 1 \\ (a/q)^{-2} \in E}}} \frac{\mu(q)^2}{\phi(q)^2\sigma(q)} \left( \frac{q}{a}\right)^3 = \frac{1}{2}\sum_{t\in \mathbb{Z}} \prod_{p \nmid t} \frac{p^2-p-1}{p^2-p}\int_E \cos\left( \frac{2\pi t}{\sqrt{y}} \right)dy.
\end{equation}
	In fact their result sums  over fixed signs in the functional equations, and would have no main term if summed over all signs. However, they obtain a nontrivial contribution when weighting by signs, therefore displaying the murmuration phenomenon, sometimes interpreted as a correlation between the signs and the coefficients of a modular form. We {establish} a similar behavior, {where the difference in the final density is due to different weighting factors}. This is {very different from} results about low-lying zeros, where the {weighting factors} do not impact the final result since they can {be} sieved out \cite{iwaniec_low_2000} {or be shown to only} contribute to the error term (the low-lying zero densities stemming from Archimedean contribution and Hecke relations alone). {This finer phenomenon can be witnessed via examining more closely averages such as equation \eqref{murmuration-sarnak}.}
	
Bober et al. \cite{bober_murmurations_2023} do assume the Generalized Riemann Hypothesis for modular forms in order to state a result with a sharp cutoff in the summation over weights; we state the result with a smooth summation and therefore do not need such an assumption. The smaller $M$ is allowed to be in the statements, the smaller the family under consideration, which translates into stronger murmuration behavior. We allow for $M \gg K^{1/3+\varepsilon}$, which improves upon $M \gg K^{5/6+\varepsilon}$ in \cite{bober_murmurations_2023}.

\section{Size of the family}

We start by examining the size of the weighted family, i.e. the denominator in \eqref{eq:thm}.
\begin{prop}
\label{prop:family-size}
{Assuming Generalized Riemann Hypothesis for Dirichlet $L$-functions. Let $K$ and $M$ be parameters satisfying $K^{1/3} \ll M \ll K^{1-\varepsilon}$ and define function $h(x) := \exp(-\tfrac{(k-K)^2}{M^2})$.} Let $E = [A,B]$ for $A < B$. We have
\begin{equation}
\label{eq:family-size}
\displaystyle\sum_{p/K^2\in E} \sum_{k \geqslant 2} h(k) \sum_{\substack{f \in H_k}} \omega_f  \log p \asymp  K^3 |E| \hat{h}(0),
\end{equation}
as $K \to \infty$.
\end{prop}

\begin{proof}

Applying Petersson trace formula \cite[Propostion 2.1]{iwaniec_low_2000} we have, for all $b \in \{0,1\}$, 
\begin{equation}
\label{eq:ptf}
\sum_{f \in H_k} \omega_f \lambda_f(p)^b = \delta_{p^b = 1} + 2\pi i^{-k} \sum_{c \geq 1} \frac{S(1,p^b;c)}{c} J_{k-1} \left( 4\pi \frac{p^{b/2}}{c} \right), 
\end{equation}
where $S(m,n;c)$ is the classical $\mathrm{GL}(2)$ Kloosterman sum, and $J_{k-1}$ is the $J$-Bessel function. 

\subsection{The diagonal contribution}

The contribution of the delta symbol in \eqref{eq:ptf} in the whole sum~\eqref{eq:family-size} is given  by
\[ \Sigma_{0,\diag} := \sum_{p/K^2 \in E} \sum_{2 \leq k \equiv 0(2)} e^{-\frac{(k-K)^2}{M^2}} i^k \log p. \]

Let $V(x)$ be a smooth function whose support lies within $[\frac{1}{4},3]$, and $V(x) = 1$ for $\frac{1}{2} \leq x \leq 2$. We look at the altered sum obtain by smoothing, 
\[ \Sigma'_{0,\diag} := \sum_{p/K^2 \in E} \sum_{2 \leq k \equiv 0(2)} e^{-\frac{(k-K)^2}{M^2}} V \left( \frac{k-1}{K} \right) i^k \log p. \]
Outside of the range $K/2 \leq k-1 \leq 2K$, {the exponential function provides rapid decay and yields a bound of size $O(K^{-A})$} for all $A>0$ since $M \ll K^{1-\varepsilon}$, and is therefore vanishingly small. Hence, we have
\[ \Sigma_{0,\diag} = \Sigma'_{0,\diag} + O(K^{-A}). \]

We execute the $p$-sum, and get a factor of $K^2|E|$ by the prime number theorem. In order to execute the $k$-sum, {we separate the sum} into residue classes modulo $4$. For those $k \equiv b$ modulo $4$, where $b \in \{0,2\}$, we get via the Poisson summation formula
\begin{align*}
  \sum_{2 \leq k \equiv b(4)} e^{-\frac{(k-K)^2}{M^2}} V \left( \frac{k-1}{K} \right) &= \sum_{m \in \mathbb{Z}} e^{-\frac{(4m+b-K)^2}{M^2}} V \left( \frac{4m+b-1}{K} \right) \\
  &= \sum_{m \in \mathbb{Z}} \int_{-\infty}^\infty e^{-\frac{(4x+b-K)^2}{M^2}}V\left( \frac{4x+b-1}{K} \right) e(-mx) \,dx \\
  &= \frac{K}{4} \sum_{m \in \mathbb{Z}} e \left( -\frac{m(1-b)}{4} \right) \int_{-\infty}^\infty e^{-\frac{(Kx+1-K)^2}{M^2}}V\left( x \right) e\left(-\frac{mK}{4}x\right) \,dx.
\end{align*}
By stationary phase, the only term that really matters is when $m = 0$, which contributes with a size of $K \hat{W}(0)$,  with the Fourier transform of $W(x) := e^{-\frac{(Kx+1-K)^2}{M^2}}V\left( x \right)$ evaluated at $0$. The other terms are {of size $O( K^{-A})$} for all $A>0$ by repeated integration by parts.
{Noting that $\hat{W}(0) \asymp \frac{M}{K}$, the} diagonal term  $\Sigma_{0, \diag}'$ therefore gives a contribution of size $K^3 |E| \hat{W}(0) \asymp MK^2 |E|$.

\subsection{Off-diagonal}
The contribution from the $c$-sum from \eqref{eq:ptf} to the whole sum~\eqref{eq:family-size} is given by
\[ \Sigma_{0,\neq} := 2\pi \sum_{p/K^2 \in E} \sum_{2 \leq k \equiv 0(2)} e^{-\frac{(k-K)^2}{M^2}} \sum_{c \geq 1} \frac{S(1,1;c)}{c} J_{k-1} \left( \frac{4\pi}{c} \right) \log p. \]
Note that $\frac{4\pi}{c} < \frac{k}{3}$ is satisfied except for a finite number of weights $k \geqslant 2$. This implies we can use the bound $J_{k-1}(x) \ll 2^{-k} x$ \cite[(2.11'{}'{}')]{iwaniec_low_2000}. We can easily see that the $c$-sum converges, and that the $k$-sum also converges. As such, it ends up with a size of $K^2|E|$, which is smaller than the diagonal term, which is of size $MK^2|E|$, finishing the proof of Proposition \ref{prop:family-size}.
\end{proof}
 
\section{Size of $\Sigma_1$}

\subsection{Summation over the weights}

The following proposition takes advantage of the summation over $k$, in the spirit of \cite[Section 8]{iwaniec_low_2000}. This adapts the approach of Xiaoqing Li \cite{li_bounds_2011}, splitting between signs. {For all $x > 0$, define}
\[ V_1(x) := \int_{-\infty}^\infty \hat{u}(v) \sin(x \cos 2\pi v) \,dv, \quad V_2(x) := \int_{-\infty}^\infty \hat{u}(v) \sin (x \sin 2\pi v) \,dv. \]

\begin{prop}
\label{prop:summation-k}
{Let $u$ be a smooth function whose support is within positive real numbers.} For $a \in \{0, 2\}$ and $x>0$, {define}
\[ S_a(x) := \sum_{2 \leq k \equiv a (4)} u(k-1) J_{k-1}(x).\]
Then we have for all $x>0$, 
\[ S_a(x) = \frac{i}{2} V_2(x) + \frac{(-1)^{a/2+1}}{2} V_1(x). \]
\end{prop}

\noindent \textit{Remark.} This is critical in the murmuration realm for the following reason: weighting by signs {only} keeps the oscillations $V_2$ (the so-called murmurations), while not doing so keeps {only} $V_1$. Indeed, we deduce from the proposition that
\begin{align}
\sum_{2 \leq k \equiv 0 (2)} u(k-1) J_{k-1}(x) & = i {V_2}(x) \\
\sum_{2 \leq k \equiv 0 (2)}  i^k u(k-1) J_{k-1}(x) & = -{V_1}(x).
\end{align}

This can be seen as a more precise version of \cite[Corollary 8.2]{iwaniec_low_2000}.

\begin{proof}
We appeal to the following integral representation of the Bessel function: 
\begin{equation}
\label{eq:bessel-integral-repr}
J_\ell(x) = \int_{-\frac{1}{2}}^{\frac{1}{2}} e(\ell t) e^{-ix \sin 2\pi t} \,dt.
\end{equation}
By viewing $\ell$ as a variable, we can say this is the Fourier transform of $e^{-ix \sin 2\pi t}$ evaluated at $-\ell$. Recall that the Fourier inversion formula gives
\[ \hat{u}(y) = \int_{-\infty}^{\infty} u(t) e(-ty) \,dt \quad \text{ and } \quad u(t) = \int_{-\infty}^\infty \hat{u}(y) e(ty) \,dy. \]


{With the given conditions, we can rewrite the sum as
\[ S_a(x) = \sum_{m \in \mathbb{Z}} u(4m+a-1) J_{4m+a-1}(x). \]
}

Using Poisson summation and inserting the integral representation \eqref{eq:bessel-integral-repr}, we have
\begin{align*}
  S_a(x) &= \sum_{m \in \mathbb{Z}} \int_{-\infty}^\infty e(-mw) u(4w+a-1) J_{4w+a-1}(x) \,dw \\
  &= \sum_{m \in \mathbb{Z}} \frac{1}{4} \int_{-\infty}^\infty e \left(-\frac{m}{4}(w-a+1) \right) u(w) J_{w}(x) \,dw \\
  &= \sum_{m \in \mathbb{Z}} \frac{e(\frac{m}{4}(a-1))}{4} \int_{-\infty}^\infty \int_{-\frac{1}{2}}^{\frac{1}{2}} e\left(-\left(\frac{m}{4}-t\right)w \right) u(w) e^{-ix \sin 2\pi t} \,dt \,dw \\
  &= \sum_{m \in \mathbb{Z}} \frac{e(\frac{m}{4}(a-1))}{4} \int_{-\frac{1}{2}}^{\frac{1}{2}} \hat{u} \left( \frac{m}{4}-t \right) e^{-ix \sin 2\pi t} \,dt.
\end{align*}
We {carry out a change of }variables, $\frac{m}{4} - t = v$, and obtain
\begin{align*}
  S_a{(x)} &= \sum_{m \in \mathbb{Z}} \frac{e(\frac{m}{4}(a-1))}{4} \int_{\frac{m}{4} - \frac{1}{2}}^{\frac{m}{4}+\frac{1}{2}} \hat{u}(v) e^{ix \sin (2\pi v - \frac{m \pi}{2})} \,dv \\
  &= \sum_{b \, \mathrm{mod} \, 4} \frac{e(\frac{b}{4}(a-1))}{4} \sum_{n \in \mathbb{Z}} \int_{n+ \frac{b}{4} - \frac{1}{2}}^{n + \frac{b}{4}+\frac{1}{2}} \hat{u}(v) e^{ix \sin (2\pi v - \frac{b \pi}{2})} \,dv \\
  &= \sum_{b \, \mathrm{mod} \, 4} \frac{e(\frac{b}{4}(a-1))}{4} \int_{-\infty}^\infty \hat{u}(v) e^{ix (\sin 2\pi v \cos \frac{b\pi}{2} - \cos 2\pi v \sin \frac{b\pi}{2})} \,dv \\
  &= \frac{1}{4} \int_{-\infty}^\infty \hat{u}(v) e^{ix \sin 2\pi v } \,dv + \frac{i^{a-1}}{4} \int_{-\infty}^\infty \hat{u}(v) e^{-ix \cos 2\pi v} \,dv \\
  &\qquad + \frac{(-1)^{a-1}}{4} \int_{-\infty}^\infty \hat{u}(v) e^{-ix \sin 2\pi v} \,dv + \frac{(-i)^{a-1}}{4} \int_{-\infty}^\infty \hat{u}(v) e^{ix \cos 2\pi v} \,dv.
\end{align*}
Hence, we have
\[ S_0{(x)} = \frac{i}{2} \int_{-\infty}^\infty \hat{u}(v) \sin (x \sin 2\pi v) \,dv - \frac{1}{2} \int_{-\infty}^\infty \hat{u}(v) \sin(x \cos 2\pi v) \,dv, \]
as well as
\[ S_2{(x)} = \frac{i}{2} \int_{-\infty}^\infty \hat{u}(v) \sin (x \sin 2\pi v) \,dv + \frac{1}{2} \int_{-\infty}^\infty \hat{u}(v) \sin(x \cos 2\pi v) \,dv. \]
This ends the proof.
\end{proof}

\subsection{Executing the sums}

Only the non-diagonal term is present in the arithmetic side of the Petersson trace formula \eqref{eq:ptf} when $b=1$. {Smoothly t}runcating the summation over $k$ with a function $V$, {the sum becomes}
\[ \Sigma_1 = 2 \pi \sum_{p/K^2 \in E} \sum_{2 \leq k \equiv 0 (2)} e^{-\frac{(k-K)^2}{M^2}} V \left(\frac{k-1}{K} \right) \log p  \sum_{c \geq 1} \frac{S(1,p;c)}{c} J_{k-1} \left( 4\pi \frac{\sqrt{p}}{c} \right). \]

%

Set $u(x) = e^{-\frac{(x+1-K)^2}{M^2}} V \left(\frac{x}{K} \right)$. We can now execute the summation over $k$ by the above section and obtain {$\sum_{k \equiv 0(2)} u(k-1)J_{k-1}(x) = iV_2(x)$} by Proposition \ref{prop:summation-k}. We obtain
\begin{align*}
  \Sigma_1 &= 2\pi i \!\! \sum_{p/K^2 \in E} \log p \sum_{c \geq 1} \frac{S(1,p;c)}{c}  V_2 \left( 4\pi \frac{\sqrt{p}}{c} \right)= 2\pi i \sum_{c \geq 1} \frac{1}{c} \int_{K^2E} V_2 \left( 4\pi \frac{\sqrt{x}}{c} \right) d\left( \sum_{\substack{p {\leq x} \\ (p,c) = 1}} S(1,p;c) \log p \right).
\end{align*}

To use the summation over $p$ of the Kloosterman sums, we appeal to the following  decorrelation lemma (see also the statement \cite[Lemma 6.2]{iwaniec_low_2000}).
\begin{lem}
Under the Generalized Riemann Hypothesis {for Dirichlet $L$-functions, we have}
\begin{equation}
\sum_{\substack{p {\leq x} \\ p \nmid c}} S(1, p, c) \log p = \frac{{x}}{\phi(c)}\mu(c)^2 + O\left(\phi(c) {x}^{1/2} \log^2 c{x}\right).
\end{equation}
\end{lem}

Therefore, we deduce that
\begin{align}
  \Sigma_1 &=  -2\pi i \sum_{c \geq 1} \frac{1}{c} \int_{K^2E} \left( \frac{x}{\phi(c)} \mu(c)^2 + O(\phi(c) x^{1/2} \log^2 x) \right) d\left(V_2 \left( 4\pi \frac{\sqrt{x}}{c} \right) \right) \\
  &= 2\pi i \int_{K^2E} \sum_{c \geq 1} \frac{\mu(c)^2}{c\phi(c)} V_2 \left( 4\pi \frac{\sqrt{x}}{c} \right) \,dx + O \left( \sum_{c \geq 1} \frac{\phi(c)}{c^2} \int_{K^2E} \log^2 x \left| V_2' \left( 4\pi \frac{\sqrt{x}}{c} \right) \right| \,dx \right).
  \label{eq:after_decorrelation}
\end{align}

{For ease of reference later, we will denote the former term as $\Sigma_{1,\main}$ and the other term as $\Sigma_{1,\err}$.} {To continue, we} look at {the following related $L$}-series, 
\begin{equation}
L(s) := \sum_{c \geqslant 1} \frac{\mu(c)^2}{\phi(c) c^s}.
\end{equation}

We have the following properties about $L(s)$, obtained by comparing Euler products on both sides of the equality.
\begin{prop}
\label{prop:L}
We have, for all $\Re(s) > -1$, 
\begin{equation}
L(s) = \frac{\zeta(s+1)}{\zeta(2s+2)} \prod_p \left( 1 + \frac{p^{-s-2}}{(1-p^{-1})(1+p^{-1-s})} \right).
\end{equation}
Moreover, $L(s)$ is meromorphic and, in the region $\Re(s)>-1/2$, has only a pole at $s=0$ with residue~$1$.
\end{prop}

\subsection{Main term {$\Sigma_{1,\main}$}}
{In order to study the first term in \eqref{eq:after_decorrelation}, we appeal to Li's expansions in \cite[Proposition~A.3]{li_bounds_2011}. }
%
We use the integral representation (A.10) therein, viz.
\[ V_2(x) = \int_{-\infty}^\infty \hat{u}(v) \sin (x \sin 2\pi v) \,dv = \frac{V_2^*(x) - V_2^*(-x)}{2i}, \]
where, for $\frac{1}{100} K \leq |x| \leq 100K$, $K^{1/3 + \varepsilon} \leq M \leq {K^{1-\varepsilon}}$ and $L_2 \geq 1$, we have the power series expansion
\begin{align}
\label{eq:V2-expansion}
  V_2^*(x) = \sum_{l = 0}^{L_2} \sum_{j =0}^l a_{j,l} \frac{x^l}{M^{5l-2j}} u_0^{(5l-2j)} \left( \frac{x-K+1}{M} \right) + O \left( \frac{|x|^{L_2+1}}{M^{3L_2+3}} + \frac{|x|}{M^7} \right),
\end{align}
where
$u_0(x) = e^{-x^2} V\left( \frac{Mx-1+K}{K} \right) = {W} \left( \frac{Mx-1+K}{K} \right)$. Note that in \cite[Proposition 5.1]{li_bounds_2011} the condition $K^{3/8 + \varepsilon} \leq M $ is assumed; however, examining the proof shows that only the condition $K^{1/3 + \varepsilon} \leq M $ is necessary. Then $u_0(\frac{x-K+1}{M}) = {W}(\frac{x}{K}) = u(x)$.
%
{The biggest reason for introducing $u_0(x)$ is that $u_0(x)$} and its derivatives are all bounded above by {$O(1)$, allowing for very nice control for its inverse Mellin transform.}

Inputing the expansion \eqref{eq:V2-expansion} into \eqref{eq:after_decorrelation}, we {are reduced to looking} at terms of form
\begin{equation}
\label{eq:mt}
\Sigma_{2,j,l} := \pi \int_{K^2E} \sum_{c \geq 1} \frac{\mu(c)^2}{c\phi(c)} \frac{\left(4\pi \frac{\sqrt{x}}{c} \right)^l}{M^{5l-2j}} u_0^{(5l-2j)} \left( \frac{4\pi \frac{\sqrt{x}}{c}-K+1}{M} \right) \,dx,
\end{equation}
and error terms of the form
\begin{equation}
\label{eq:et}
\Sigma_{3} := \int_{K^2E} \left| \sum_{c \geq 1} \frac{\mu(c)^2}{c\phi(c)} \cdot O \left( \frac{|\frac{4\pi \sqrt{x}}{c}|^{L_2+1}}{M^{3L_2+3}} + \frac{|\frac{4\pi \sqrt{x}}{c}|}{M^7} \right)  \right| \,dx,
\end{equation}
their relation with $\Sigma_1$ being
\[ \Sigma_1 = \sum_{l=0}^{L_2} \sum_{j=0}^l a_{j,l} \Sigma_{2,j,l} + O(\Sigma_3). \]

\subsubsection{Main terms {$\Sigma_{2,j,l}$}}

{W}e define $U_0(y) := u_0 \left( \frac{y-K+1}{M} \right)$. Then $U_0^{(k)}(y) = M^{-k} u_0^{(k)} \left( \frac{y-K+1}{M} \right)$, which means $U_0^{(k)}(y) \ll M^{-k}$.
Let $\tilde{U}_0(s)$ to be the Mellin transform of $U_0$. Then from the {sizes of derivatives of $U_0$}, we see that $\tilde{U}_0(s) \ll (1+|t|)^{-\sigma} M^{-\sigma}$, for $\sigma \in \mathbb{R}_{\geq 0}$.

Note that the Mellin transform of $U_0^{(k)}$ is $(-1)^{k} \frac{\Gamma(s)}{\Gamma(s-k)} U_0(s-k)$.
Hence, for $\gamma < 1+l$, we obtain that the contribution of \eqref{eq:mt} is
\begin{align}
 \Sigma_{2,j,l} = &\pi \int_{K^2E} \sum_{c \geq 1} \frac{\mu(c)^2}{c\phi(c)} \frac{\left(4\pi \frac{\sqrt{x}}{c} \right)^l}{M^{5l-2j}} u_0^{(5l-2j)} \left( \frac{4\pi \frac{\sqrt{x}}{c}-K+1}{M} \right) \,dx \notag \\
  =& \pi \int_{K^2E} \sum_{c \geq 1} \frac{\mu(c)^2}{c\phi(c)} \left(4\pi \frac{\sqrt{x}}{c} \right)^l  U_0^{(5l-2j)} \left( 4\pi \frac{\sqrt{x}}{c}\right) \,dx \notag \\
  =& \pi (-1)^{5l-2j} \cdot \frac{1}{2\pi i} \int_{(\gamma)} \int_{K^2 E} L(1+l-s) (4\pi \sqrt{x})^{l-s} \frac{\Gamma(s)}{\Gamma(s-5l+2j)} \tilde{U}_0(s-5l+2j) \,dx \,ds. \label{eq:moved-integral-mt}
\end{align}

We move the line of integration to ${\Re s =  } \frac{3}{2}+l$, and the {moved} integral contributes as an error term {due to the strong} decay of $\tilde{U}_0$. The residue at $s = 1+l$ is
\begin{align*}
  \frac{(-1)^{5l-2j}}{4} \int_{K^2 E} x^{-1/2} \frac{\Gamma(1+l)}{\Gamma(1-4l+2j)} \tilde{U}_0(1-4l+2j) \,dx.
\end{align*}
Note that $1-4l+2j \leq 1 - 4l + 2l = 1-2l \leq 0$ as long as $l \geq 1$, therefore implying that the only contributing term  in \eqref{eq:V2-expansion} is the one corresponding to $l = 0$. {In this case, we have $l=j=0$ and the contribution is
\begin{align*}
  \frac{1}{4} \int_{K^2 E} x^{-1/2} \,dx \cdot \tilde{U}_0(1) = \frac{K(\sqrt{B}-\sqrt{A})}{2} \int_0^\infty U_0(y) \,dy &= \frac{K(\sqrt{B}-\sqrt{A})}{2}\int_{-\infty}^\infty W \left( \frac{y}{K} \right) \,dy \\ &= \frac{K^2(\sqrt{B}-\sqrt{A})}{2}\hat{W}(0).
\end{align*}}
This provides the main term $K^2 (\sqrt{B} - \sqrt{A}) {\hat{W}}(0)$ in Theorem \ref{thm:thm}.

\subsubsection{Error term {$\Sigma_3$}}

For the term \eqref{eq:et} arising from the big-O term in \eqref{eq:V2-expansion}, the sum over $c$ is already restricted to $\frac{4\pi\sqrt{x}}{100K} \leq c \leq \frac{400\pi\sqrt{x}}{K}$ due to the {properties} of $V_2$ as stated in {\cite[Proposition~A.3]{li_bounds_2011}}.  Recall that $K^2 A \leq x \leq K^2B$, so
$0.04\pi \sqrt{A} \leq c \leq 400\pi \sqrt{B}$, and the $c$-sum is therefore of constant length. The  error term hence gives a size of $O( \frac{K^{2 + L_2 + 1}}{M^{3L_2+3}} + \frac{K^3}{M^7})$. To ensure it is negligible compared to the main term {$K^2\hat{W}(0) \asymp MK$}, we require
\begin{align*}
K^{\frac{1}{3} + \frac{2}{3(4+3L_2)}} \ll M.
\end{align*}
This is ensured by the assumption $M \geqslant K^{1/3+\varepsilon}$.

\subsection{Error term {$\Sigma_{1,\err}$}}
We need to estimate the remaining term in \eqref{eq:after_decorrelation}, i.e.
\begin{equation}
\label{eq:error}
{\Sigma_{1,\err} =} \sum_{c \geq 1} \frac{\phi(c)}{c^2} \int_{K^2E} \log^2 x \left| V_2' \left( 4\pi \frac{\sqrt{x}}{c} \right) \right| \,dx.
\end{equation}

The expansion of $(V_2^*)'(x)$ is that of $V_2^*(x)$ differentiated term-by-term, and the error term is going to be one power of $x$ less, i.e.
\begin{align*}
\label{eq:V2star-expansion}
  (V_2^*)'(x) &= U_0'(x) + \sum_{l = 1}^{L_2} \sum_{j =0}^l l a_{j,l} x^{l-1}U_0^{(5l-2j)} (x) + \sum_{l = 1}^{L_2} \sum_{j =0}^l a_{j,l} x^l U_0^{(5l-2j+1)} (x) + O \left( \frac{|x|^{L_2}}{M^{3L_2+3}} + \frac{1}{M^7} \right).
\end{align*}

The big-O terms lead to bound of size $O(K^{2+L_2} M^{-(3L_2+3)})$ and $O(K^2 M^{-7})$. {For these to truly be error terms, we require} $K^{\frac{1}{3} - \frac{1}{3(4+3L_2)}} \ll M$ and $K^{1/8} \ll M$, {which is guaranteed} by the condition $M \gg K^{1/3+\varepsilon}$.

{The} pieces we are looking at are for $l \geq 1$,
\[ \sum_{c \geq 1} \frac{\phi(c)}{c^2} \int_{K^2E} \log^2 x \left|  \left( 4\pi \frac{\sqrt{x}}{c} \right) ^{l-1}  U_0^{(5l-2j)} \left( 4\pi \frac{\sqrt{x}}{c} \right) \right| \,dx, \]
and for $l \geq 0$, 
\[ \sum_{c \geq 1} \frac{\phi(c)}{c^2} \int_{K^2E} \log^2 x \left|  \left( 4\pi \frac{\sqrt{x}}{c} \right) ^{l}  U_0^{(5l-2j+1)} \left( 4\pi \frac{\sqrt{x}}{c} \right) \right| \,dx. \]
We first change variables so that $4\pi \frac{\sqrt{x}}{c} = w$. Then our integration range becomes $[4\pi K\frac{\sqrt{A}}{c},4\pi K\frac{\sqrt{B}}{c}]$, and $x = \frac{w^2 c^2}{16\pi^2}$, and  $dx = \frac{wc^2 \,dw}{8\pi^2} $.
Then the pieces we are interested in become, for $l \geq 1$, 
\[ \sum_{c \geq 1} \phi(c) \int_{\frac{4\pi K}{c} \sqrt{E}} \log^2 (wc^2) \left| w^l  U_0^{(5l-2j)} (w) \right| \,dw \]
and for $l \geq 0$, 
\[ \sum_{c \geq 1} \phi(c) \int_{\frac{4\pi K}{c} \sqrt{E}} \log^2 (wc^2) \left|  w^{l+1}  U_0^{(5l-2j+1)} (w) \right| \,dw. \]

{Note that $u_0(x) = U_0( Mx+K-1)$.} Changing variables to $w = Mx+K-1$, i.e. $x = \frac{w+1-K}{M}$, these two pieces become, for $l \geq 1$, 
\[ M \sum_{c \geq 1} \phi(c) \int_{\frac{\frac{4\pi K}{c} \sqrt{E} + 1 - K}{M} } \log^2 ((Mx+K-1)c^2) \left| \frac{(Mx+K-1)^l}{M^{5l-2j}}  u_0^{(5l-2j)} (x) \right| \,dx \]
and for $l \geq 0$, 
\[ M \sum_{c \geq 1} \phi(c) \int_{\frac{\frac{4\pi K}{c} \sqrt{E} + 1 - K}{M}} \log^2 ((Mx+K-1)c^2) \left| \frac{(Mx+K-1)^{l+1}}{M^{5l-2j+1}}  u_0^{(5l-2j+1)} (x) \right| \,dx. \]
Recall $u_0(x) = e^{-x^2} V(\frac{Mx-1+K}{K})$, {which implies its $k$-th derivative is effectively of size} $e^{-x^2}$, thus we just need to look at a small interval $x \in [-M^{\varepsilon},M^{\varepsilon}]$. Recall that the $c$-sum is effectively a sum of constant size. As such, we obtain the bound $M K^{l+1} M^{-(5l-2j+1)} = M^{-(5l-2j)} K^{l+1} \ll M^{-3l} K^{l+1}$. To ensure this is negligible compared to the main term, we require $M^{-3l} K^{l+1} \ll MK$ which means
\begin{align*}
 K^{\frac{1}{3} - \frac{1}{3(1+3l)}} \ll M.
\end{align*}
This is also ensured by the condition $M \geqslant K^{1/3+\varepsilon}$, ending the proof of Theorem \ref{thm:thm}.

\nocite{*}

\bibliographystyle{acm}
\bibliography{biblio}

\end{document}